\documentclass[12pt]{amsart}

\usepackage{graphicx,color}
\usepackage{latexsym}
\usepackage{graphicx} 
\usepackage{amsthm,amsmath, amssymb,amsfonts,esint}
\usepackage{layout,textcomp}

\usepackage{hyperref}

\usepackage[english]{babel}
\theoremstyle{plain}

\newcommand{\tr}{\mbox{tr}\;}

\newtheorem{theorem}{Theorem}[section]

\newtheorem{lemma}[theorem]{Lemma}

\newtheorem{corollary}[theorem]{Corollary}

\newtheorem{definition}[theorem]{Definition}

\newtheorem{proposition}[theorem]{Proposition}

\newtheorem{remark}[theorem]{Remark}

\numberwithin{equation}{section}

\begin{document}

\title[On the prescribed $Q$-curvature problem]{ On the prescribed $Q$-curvature problem in Riemannian manifolds}



\author{ Fl\'avio F. Cruz }
\address{
Current: Institut de Math\'ematiques de Jussieu, Universit\'e Paris Diderot, B\^atiment Sophie Germain, Paris 7, 75205 Paris Cedex 13, France. 
Permanent: Departamento de Matem\'atica, Universidade Regional do Cariri, Campus Crajubar\\
 Juazeiro do Norte, Ceara, CE - Brazil 63041-141, Brazil}
\email{ flavio.franca@urca.br}

\author{ Tiarlos Cruz}
\address{
Universidade Federal de Alagoas\\
Instituto de Matemática\\
Maceió, AL -  57072-970, Brazil}
\email{cicero.cruz@im.ufal.br}

\begin{abstract}
We prove the existence of metrics with prescribed $Q$-curvature under natural assumptions on the sign of the prescribing function
and the background metric. In the dimension four case, we also obtain existence results for curvature forms
requiring only restrictions on the Euler characteristic. Moreover, we derive a prescription result for open submanifolds
which allow us to conclude that any smooth function on $\mathbb{R}^n$ can be realized as the $Q$-curvature of a Riemannian metric.
\end{abstract}

\maketitle
\section{Introduction}

The problem of prescribing Riemannian curvatures has attracted considerable attention in the last decades.
Such a problem provides an interesting interplay between differential geometry and nonlinear partial differential equations, since it can relies on to solve  a system of PDE on the fundamental tensor of the Riemannian metric. 

In this paper we are interested at the problem of prescribing a well known fourth order conformal invariant introduced by Tom Branson \cite{B} called 
$Q$-curvature. For surfaces, the $Q$-curvature is the half of the scalar curvature and for conformally flat manifolds of dimension four its integral is a multiple of the Euler characteristic, that obviously refers to  Gauss Bonnet Theorem. 
The $Q$-curvature also shares the same conformal behaviour  that the scalar curvature and satisfies  analogous transformations laws under conformal rescaling of the metric.
Its worth to mention that this scalar invariant  has also been studied in theoretical physics with applications  in quantum field theory and higher-derivative field theories (for more details,  c.f.  \cite{LO, N}).

Kazdan and Warner have shown in \cite{KW,KW3} that the Euler characteristic sign condition
given by the Gauss Bonnet Theorem  is necessary and sufficient for a smooth function  on a given compact 2-manifold to be the
Gaussian curvature of some metric. In arbitrary dimension, the problem of prescribing scalar curvature was solved in \cite{KW2,KW3} requiring the existence of metrics
with constant scalar curvature and using  conformal deformation techniques. Considering the analogy between scalar curvature and $Q$-curvature,  it is reasonable to ask whether  those results can be generalised to $Q$-curvature.
Building upon to the methods developed by Kazdan-Warner \cite{KW,KW1,KW3}  we prescribe the $Q$-curvature under natural assumptions on the sign condition of the prescribing function.

\begin{theorem}\label{princ1}
Let $(M^n, g_0)$ be a compact Riemannian n-manifold ($n\geq 3$) with $Q$-curvature $Q_{g_0} = Q_0,$ where $Q_0$ is a constant. If $Q_0\neq0,$ then any smooth function $f$
having the same sign as $Q_0$ somewhere, is the $Q$-curvature of some metric. If
$Q_0\equiv 0,$ then any  smooth function $f$ that changes sign is the $Q$-curvature of some metric.
\end{theorem}

We recall that Gauss-Bonnet formula says that for any  compact
 surface $M$, the total Gaussian curvature of $M$  is equal to $2\pi\chi(M)$, where $\chi(M)$ is the Euler characteristic of $M.$ As  already mentioned, in a four dimensional Riemannian manifold $(M^4,g)$
the $Q$-curvature satisfies a similar formula. Precisely,
\begin{equation}\label{gb}
\int_M\Big(Q_g+\frac{|W_g|^2}{4}\Big)dv=8\pi^2\chi(M),
\end{equation}
where $W_g$  stands  for the Weyl tensor of $g.$ As $W_g$, the total $Q$-curvature, denoted by
$\kappa_P=\int_MQ_gdv,$
 is invariant under conformal changes.

Analogous to the uniformization theorem for compact surfaces,  there is a  four dimensional  version result involving the $Q$-curvature, which was proved by  Djadli and Malchiodi in \cite{DM}. They show that if $\kappa_P\neq 8k\pi^2$ for $k = 1, 2,\ldots $ and   $\ker P_g = \{const\},$ where $P_g$ is the \textit{Paneitz-Operator}, then $(M,g)$ admits a conformal metric  with constant $Q$-curvature.
Using this  existence result  one can prove the following   converse to the Gauss Bonnet Theorem for locally conformally flat manifold, saying that  \eqref{gb}  imposes a sign condition on $Q$ depending on $\chi(M),$ and conversely. 

\begin{corollary}\label{corprinc}
Let $(M^4,g)$ be a compact locally conformally 
flat $4$-manifold such that  $\ker P_g = \{const\}$ and $\kappa_P\neq 8k\pi^2$ for $k = 1, 2,\ldots.$ Then, a smooth function $f$ on $M$ is the $Q$-curvature of some metric on $M$ iff
\begin{itemize}
\item[a)] $f$ is positive somewhere, if $\chi(M)>0$;
\item[b)] $f$ changes sign or $f\equiv 0$, if $\chi(M)=0$;
\item[c)] $f$  is negative somewhere, if $\chi(M)<0.$
\end{itemize}
\end{corollary}

Another interesting question that seems natural is to ask if  a given function defined on a non-compact manifold is the curvature of some Riemannian metric. Related to this question, we proved the following result for open submanifolds of compact manifolds.

\begin{theorem}\label{noncomp}
Let $(M^n, g)$ be a non-compact Riemannian manifold, $n\geq 3,$  diffeomorphic to an open submanifold of some compact  manifold $(N^n,h)$ of constant $Q$-curvature $Q_0\neq0.$ Any smooth function $f$ on $M^n$ can be realized as the $Q$-curvature of some Riemannian metric on $M.$
\end{theorem}

In particular, we obtain:

\begin{corollary}
Any $Q\in C^{\infty}(\mathbb{R}^n),\;n\geq 3,$   is the $Q$-curvature of some Riemannian metric on $\mathbb{R}^n.$ 
\end{corollary}

  We should mention that  is possible to rephrase the problem of prescribing curvature depending of the Euler characteristic in terms  of curvature forms.  Recall that the generalized Gauss-Bonnet theorem says that 
$$
\frac{1}{16}\int_M\mbox{Pfaff}=8\pi^2\chi(M),
$$
where $M$ is a $4$-dimensional compact orientable Riemannian manifold  boundaryless and $\mbox{Pfaff}$ is the Pfaffian  $4$-form. We wonder to know if the conversely is true, that is, given any $4$-form $\Omega$ satisfying $\int_M\Omega=8\pi^2\chi(M),$  there exists a metric $g$ on $M$ such that $\Omega=\mbox{Pfaff}$?
We give an afirmative answer to this question with a result that is the analogue of the main result of Wallach-Warner \cite{WW}. It should be emphasized that this problem for higher dimensions was  posed 
in \cite{KBOOK}, p. 3, and here we  solve it in dimension four  just for a certain class of manifolds.

\begin{theorem}\label{curvform}
Let $(M,g)$ be a compact, connected, orientable Riemannian $4$-manifold such that  $\ker P_g = \{const\}.$  Given any $4$-form $\Omega$  
that satisfies $$
\int_{M}\Omega=8\pi^2\chi(M),
$$
then there exists a metric pointwise conformal to $g$ such that $\Omega$  is a curvature $4$-form.
\end{theorem}

The paper is organized  as follows.
In Section \ref{preli}, we establish the fundamental concepts and prove a local surjectivity result for the $Q$-curvature map. In Section \ref{compactopen}, we prove Theorem \ref{princ1} and Theorem \ref{noncomp}. Theorem \ref{curvform} is proved in Section \ref{4-forms}.

\section{ Local surjectivity}\label{preli}

Throughout  this section, $M^n$ will denote a  compact  connected  Riemannian manifold without boundary, $n\geq 3,$ $S_2(M)$  the set of all smooth symmetric 2-tensors on $M,$ and 
$\mathcal{M}^{4,p}$  the space  of class $W^{4,p}$ of  the symmetric $(0, 2)$-tensors.
The $Q$-curvature of order four, denoted by $Q_g$,  is defined as
\begin{equation}
\label{Qcurv} Q_g = a_n \Delta_g R_g +b_n |Ric_g|_{g}^2 + c_n R_g^2,
\end{equation}
where $a_n=-\frac{1}{2(n-1)},$ $b_n=-\frac{2}{(n-2)^2},$ $c_n=\frac{n^2(n-4)+16(n-1)}{8(n-1)^2(n-2)^2},$ $\Delta_g := g^{ij} \nabla_i \nabla_j$, $R_g$ is  the scalar curvature and $|Ric_g|$ is the norm of the Ricci tensor.

Now consider the following nonlinear fourth order differential operator 
$$
Q:\mathcal{M}^{4,p}\to L^p(M),\quad \quad q\mapsto Q_g
$$
 where  $\mathcal{M}^{4,p}$ denotes the open subset of $S^{4,p}_2$ of the Riemannian metrics on $M.$ It is possible to show, using multiplicative properties of Sobolev spaces, see for example \cite{MS}, the $Q$-curvature map is well  defined and smooth for $2p>n.$ In order to study the local surjectivity  of the $Q$-curvature, we have  to study the kernel of $L^2$-formal adjoint for the linearization of $Q$-curvature.

Before stating results giving the linearization and $L^2$ formal adjoint of the map $q\mapsto Q_g,$  we first need a few definitions. 
The \textit{Lichnerowicz Laplacian} acting on $h\in S_2(M)$ is defined to be
\begin{align*}
\Delta_L h_{jk} = \Delta h_{jk} + 2 (\overset{\circ}{Rm}\cdot
h)_{jk} - R_{ji} h^i_k - R_{ki}h^i_j,
\end{align*}
where $(\overset{\circ}{Rm} \cdot h )_{jk}:= R_{ijkl} h^{il}.$

Following  notation in \cite{LY} we have

\begin{proposition}[Lin-Yuan, \cite{LY}] Given an infinitesimal variation $h,$  the linearization of the $Q$-curvature $Q$ at $g,$ denoted by $L_g,$ in the direction of $ h$ is given by

\begin{align}\label{Q_linearization}
L_g h =&\;a_n [ - \Delta^2 (\tr h) +  \Delta \delta^2 h + \frac{1}{2} dR \cdot (d( \tr h ) + 2\delta h) 
\\ \notag &- \Delta ( Ric \cdot h )- \nabla^2 R\cdot h] - b_n [ Ric \cdot \Delta_L h + Ric \cdot
\nabla^2(tr h)   \\
\notag &+ 2 Ric \cdot\nabla (\delta h) +2 [Ric\times Ric) \cdot h
]+ 2 c_n R( - \Delta (tr h) \\ \notag
 &+  \delta^2 h - Ric \cdot h],
\end{align} 
where $\nabla^2=\nabla_i\nabla_j,$    $(\delta_g h)_i := - (div_g h)_i = -\nabla^j h_{ij}$ and $(Ric\times Ric)_{ij}:=R^{l}_iR_{lj}.$
\end{proposition}

The next theorem address  the $L^2$ formal adjoint, denoted by $L_g^*.$

\begin{proposition}[Lin-Yuan\cite{LY}]  The $L^2$
formal adjoint of $L_g$ is given by
\begin{align}\label{Q_adjoint}
L_g^* f =&\; a_n [  - g \Delta^2 f + \nabla^2
\Delta f - Ric \Delta f + \frac{1}{2} g \delta (f dR) + \nabla ( f
dR)  \\ 
\notag &- f \nabla^2 R ] - b_n [ \Delta (f Ric) + 2 f
\overset{\circ}{Rm}\cdot Ric + g \delta^2 (f Ric) \\ \notag
&+ 2 \nabla \delta (f Ric) ]- 2 c_n ( g\Delta (f R) - \nabla^2 (f R) + f R
Ric ).
\end{align}
\end{proposition}
 
 Recall that the principal symbol of a differential operator  is an invariant that captures some very strong properties of the operator,  as 
example, the ellipticity. In our case, it was observed  in \cite{LY} that the  principal symbol of $L_{g}^*$ is
\begin{align*}
\sigma_{\xi}(L_{g}^*) = - a_n \left( g |\xi|^2 - \xi \otimes\xi \right)|\xi|^2.
\end{align*}
Notice that $L_{g}^*$ has an injective  symbol. Indeed, taking the trace we obtain  
\begin{align*}
tr\ \sigma_{\xi}(L_{g}^*) = - a_n \left(n-1 \right)|\xi|^4.
\end{align*}
If  $\sigma_{\xi}(L_{g}^*)$ were zero, then $tr\ \sigma_{\xi}(L_{g}^*) $ would be zero with $\xi \neq 0$. 
Therefore, $L_{g}^*$ is overdetermined elliptic and, thus,  $L_gL_g^*$ is elliptic. This fact plays a fundamental role in the proof of Theorem \ref{ITF}.

Following  Chang-Gursky-Yang  \cite{CGY} we  define the  notion of $Q$-singular space.

\begin{definition}[Chang-Gursky-Yang \cite{CGY}]
A complete Riemannian manifold $(M, g)$ is said to be $Q$-singular, if  $ L^*_g$ possesses non-trivial kernel, that is, $$\ker L^*_g \neq \{ 0 \}.$$  In this case, we say that $(M, g, f)$ is a $Q$-singular space,  where   $f (\not\equiv 0)$ is in the kernel of $L_g^*.$
\end{definition}

Taking the trace of (\ref{Q_adjoint}) one obtain
\begin{align*}
 tr_g L_g^* f = \frac{1}{2} \left( P_g - \frac{n + 4}{2} Q_g \right) f, 
\end{align*}
which allows to prove that  the condition  of non-$Q$-singularity is satisfied for generic metrics.

\begin{theorem}[Chang-Gursky-Yang \cite{CGY}]\label{cgy}
Suppose that $(M^n, g)$ is a $Q$-singular space, then it  has constant $Q$-curvature and 
\begin{align*}
\frac{n+4}{2}Q_g \in Spec(P_g).
\end{align*}
\end{theorem}

\begin{remark}
It is possible to show that Theorem \ref{cgy} implies that the set of non-$Q$-singular metrics on $M$ is open and dense in the
$W^{4,p}$ topology for any $2p > n.$
\end{remark}

We are now prepared to prove the following proposition about the local surjectivity of the $Q$-curvature map. However, before we do this, we will briefly discuss  some useful facts.

Let $E, F$ be vector
bundles over $M,$ and let  $D: W^{s,p}(E) \to W^{s-k,p}(F)$ be  a k-th order differential operator, where $k\leq s\leq \infty,$ $1<p<\infty.$ We notice that we can make  use of a splitting lemma of Berger and Ebin \cite{BE}. Recall that if 
$D^*:W^{s+k,p}(F)\to W^{s,p}(E),$ its $L^2$ formal adjoint, has injective symbol, then 
$$
W^{s,p}(F)=\mbox{Im}\;D\oplus \mbox{ker}\;D^*.
$$
A useful consequence is  the following: If $D^*$  is injective and has injective symbol then we
can conclude that $D$ is surjective.

\begin{theorem}\label{ITF}
Let  $f\in L^p(M)$ and $2p>n$. Assume that  $(M, g_1)$ is not $Q$-singular.
Then there is an $\eta>0$  such that if $$\|f-Q_{g_1}\|_p<\eta,$$ then there is a $g\in \mathcal{M}^{4,k}$
such that $Q_g=f.$ Furthermore, $g$ is smooth if $f$ is smooth.
\end{theorem}

\begin{proof}
The map $g\mapsto Q_g$ is a quasilinear differential operator of fourth order which can be  extend from  $\mathcal{M}^{4,p}$ into $L^p,$ for $2p>n,$  using the Sobolev  Embedding Theorem. 
Let $F$ be a map from  a sufficiently small neighborhood 
of zero $U\subset W^{8,p}$ into $L^p$ given by $$F(u)=Q_{g_1+L^*u}.$$
We will use the  implicit function theorem for Banach spaces in order to solve  this eighth order quasilinear elliptic equation. First, it is straightforward to see that 
 $F'(0)$ is elliptic at $u=0$ since $F'(0)v= L_{g_1}L^*_{g_1}v.$  Further, it is invertible since  $ker L_{g_1}L^*_{g_1}=ker L^*_{g_1}=0.$ Indeed, note that in $L^2$ 
$$
\langle L_{g_1}L^*_{g_1}v,v\rangle=\| L^*_{g_1}v\|^2=0.
$$
It follows from  the implicit function theorem  in Banach spaces that $F$ maps a
neighborhood of zero in $W^{8,p}$ onto an $L^p$ neighborhood of $f$. The final assertion of the theorem follows from  elliptic
regularity theory.
\end{proof}

\section{Prescribing curvature on compact and open manifolds}\label{compactopen}

In this section we will prove the results that establish conditions to prescribing the $Q$-curvature.  In other words, we
present conditions that allow us to find a solution $g$ for the four order differential equation
\begin{equation}\label{qualquer}
Q_g=f,
\end{equation}
for a given smooth function $f.$ 
 First, we fix a non-$Q$-singular metric $g_1$ and set $Q_{g_1}=Q_1.$  The idea is as follows. In order to solve \eqref{qualquer}, we apply Theorem \ref{ITF}, which holds just for functions $f$  near $Q_{g_1}$ in some appropriated sense, that is,   $\|Q_1-f\|<\varepsilon$ in $L^p$ norm. To overcome this difficult we make use of the  existence of a  diffeomorphism  $\varphi$ such that $\|Q_1-f\circ\varphi\|_p<\varepsilon$ (see  Lemma \ref{approx}). Hence,  there exists a metric  up to diffeomorphism that solves \eqref{qualquer} for $f\circ\varphi$.
 
 We also refer the reader to \cite{BFR, Br, Br2, CR, DR, MST} and  references therein. In these papers, \eqref{qualquer} is solved using a metric $g$ that is pointwise conformal to a fixed metric $g_0,$ say $g=e^{\varphi}g_0,$ for some function $\varphi$. 	In this case, \eqref{qualquer} takes the form $Q(e^{u}g_0)=f$.

We will make use of the following  Approximation Lemma due to Kazdan and Warner \cite{KW1,KW3}.

\begin{lemma}[Approximation Lemma \cite{KW1,KW3}] \label{approx}
Let $M^n$ be a compact manifold with $dim\; M = n\geq 2$ and let $f,g\in C^1(M)\cap L^p(M).$ If there exists a positive constant $c$ such that the range of $g$ is contained in the range of $f$,that is, $$\inf cf\leq
g(x)\leq \sup cf$$ for almost all $x$ on $M,$ then given any $\varepsilon > 0$ there is a diffeomorphism $\varphi$ of $M$ such that
$$\|f\circ\varphi-g\|_p<\varepsilon,$$ and conversely.
\end{lemma}

\begin{remark}
The above result is trivially false for the uniform metric.
\end{remark}

 Now we can prove our first prescribing result.

\begin{proposition}\label{propimp}
Let $(M^n, g_0),$ $n\geq 3$, be a smooth compact Riemannian manifold with $Q$-curvature, $Q$, and let $f\in C^{\infty}(M).$  Assume that there is a positive constant $c$ such that
\begin{equation}\label{ineq}
\min cf<Q(x)<\max cf
\end{equation}
for all $x\in M,$ then there is a smooth metric $g$ with $Q_g=f.$
\end{proposition}

\begin{proof}
Assume that $(M, g_0)$ is non-$Q$-singular. By  Lemma \ref{approx}, there exists a diffeomorphism $\varphi$ such that $$\|Q-cf\circ\varphi\|_p<\varepsilon$$ for all $\varepsilon>0$ and $2p> n.$  Since  $ker L^*_g \neq \{ 0 \},$  it  follows from Theorem \ref{ITF}  that there is a metric $g_1$ with $Q_{g_1}=cf\circ\varphi.$ Since the $Q$-curvature is invariant by diffeomorphism  the metric given by $g=(\varphi^{-1})^*(\sqrt{c}g_1)$ has 	$Q$-curvature $f$ as desired.
	Otherwise, if $M$ is $Q$-singular (which implies that $Q$ is  constant) we may modify slightly $g_0$ in order to obtain a metric $\tilde g$ with non-constant $Q$-curvature still satisfying \eqref{ineq} and the result follows. 
\end{proof}

As a consequence,  we  proof the following  result that corresponds to Theorem  \ref{princ1}. 

\begin{theorem}[Theorem \ref{princ1}]
Let $(M^n, g_0)$ be a compact, Riemannian, n-manifold ($n\geq 3$) with $Q$-curvature $Q_{g_0} = Q_0,$ where $Q_0$ is a constant. If $Q_0\neq0,$ then any function f
having the same sign as $Q_0$ somewhere is the $Q$-curvature of some metric, while if
$Q_0\equiv0,$ then any function f that changes sign is the $Q$-curvature of some metric.
\end{theorem}

\begin{proof}
This is an immediate consequence of Proposition \ref{propimp}. Indeed, note that \eqref{ineq} is satisfied by any function $f$ having the same sign
as $Q_0$ at some point of $M$, moreover if $Q_0$ is identically zero, then \eqref{ineq} is satisfied if $f$ changes
sign on $M.$
\end{proof}

\begin{remark}
 By elliptic regularity, instead $f\in C^{\infty}(M),$ we could assume  $f\in C^{j+\alpha}(M),$ for some $j\in\mathbb{N}$ and $\alpha\in (0,1).$ Notice that such a assumption would imply that the found metric $g\in C^{j+\alpha+4}.$
\end{remark}

More recently, Lin and Yuan \cite{LY}  have shown  that non-$Q$-singular spaces  are  linearized stable, which turns to be very useful to finding solution in a given  
direction. Thus, it is possible  to prescribe some kinds of $Q$-curvature problems. To be more precise,  they proved that any smooth function can be realized as a $Q$-curvature on  non-$Q$-singular spaces with vanishing $Q$-curvature.

The problem concerning the existence of metrics of constant $Q$-curvature  in compact 4-manifolds was developed by Chang and Yang \cite{CY}, Gursky \cite{G} and Wei and Xu \cite{WX},  and more recently,  Djadli and Malchiodi \cite{DM}    provided  extensions  of these works.
In dimension four we have the following result whose  assumptions are conformally invariant and generics (see also \cite{ND}, for  dimension higher than four).
\begin{theorem}[Djadli and Malchiodi \cite{DM}]\label{DjaMalc}
 Suppose $ker P_g = \{const\}$, and assume that $\kappa_P\neq 8k\pi^2$ for $k = 1, 2, \ldots .$ Then $(M,g)$ admits a conformal metric with constant $Q$-curvature.
\end{theorem}

We obtain the following  corollary of Theorem \ref{princ1} for compact locally conformally flat (or l.c.f) manifolds  of dimension four,  whose notion is characterized by the Weyl tensor.

\begin{corollary}[Corollary \ref{corprinc}]
Let $(M^4,g)$ be a compact locally conformally 
flat $4$-manifold such that  $\ker P_g = \{const\}$ and $\kappa_P\neq 8k\pi^2$ for $k = 1, 2,\ldots.$ Then, a smooth function $f$ on $M$ is the $Q$-curvature of some metric on $M$ iff
\begin{itemize}
\item[a)] $f$ is positive somewhere, if $\chi(M)>0$;
\item[b)] $f$ changes sign or $f\equiv 0$, if $\chi(M)=0$;
\item[c)] $f$  is negative somewhere, if $\chi(M)<0.$
\end{itemize}
\end{corollary}

\begin{proof}
Since  the existence of metrics with constant $Q$-curvature is given by Theorem \ref{DjaMalc}, we can see that the sign condition of given functions depend on the sign  of the Euler characteristic by \eqref{gb}.  
\end{proof}

\begin{remark}
  $\kappa_P$ has an upper sharp inequality  (see \cite{G}) besides  being  multiple of the  Euler characteristic  on conformally flat structures.
\end{remark}

Next we prescribe the $Q$-curvature of open submanifolds  which  reads as follows.

\begin{theorem}[Theorem \ref{noncomp}]\label{noncompa2}
Let $M^n$ be a non-compact Riemannian manifold, $n\geq 3,$  diffeomorphic to an open submanifold of some compact  n-manifold $N$ of constant $Q$-curvature $Q_0\neq0.$ Then every  $f\in C^{\infty}(M)$ is the $Q$-curvature of   some Riemannian metric on $N.$
\end{theorem}

\begin{remark}
For surfaces, a version of Theorem \ref{noncompa2} was proved by Kazdan and Warner  \cite{KW4} using conformal deformation methods (different from our proof in several steps). 
Recall that on surfaces the $Q$-curvature  is essentially the Gaussian curvature. By Bonnet-Myers theorem and completeness, if the sign of $f$ were positive, then we would have compactness. Hence we conclude that 
one cannot always hope to achieve a complete metric which has a given $f.$
\end{remark}

\begin{proof}[Proof of Theorem \ref{noncompa2}] 

Assume with no loss of generality that $N\setminus M$ contains an open set and that $M$ and $N$ are connected. 
Now, extend $Q$ to $N$ by defining it to be identically equal to $Q_0$ on $N\setminus M.$ By  Approximation Lemma \ref{approx},  there exists a diffeomorphism $\varphi$ on $N$  such that 
$$
\|f\circ\varphi-Q_0\|_p<\varepsilon,
$$
where $2p> n.$ 
Since $\varphi^{-1}(M)\subset N,$ by Theorem \ref{ITF}  there is a metric $g_1$ with $$Q_{g_1}=f\circ\varphi.$$ Hence $f$ is a curvature of the pulled-back metric $(\varphi^{-1})^*(g_1)$ on $M$ and the result follows.
\end{proof}

An immediate and interesting consequence of the above theorem is the following.

\begin{corollary}
Any $Q\in C^{\infty}(\mathbb{R}^n),\;n\geq 3,$   is the $Q$-curvature of some Riemannian metric on $\mathbb{R}^n.$ 
\end{corollary}

\section{prescribing 4-forms}\label{4-forms}

Given a Riemannian manifold $(M^{n},g)$ let us recall some basic facts concerning its Riemannian geometry. The first one deals with the classical decomposition of the Riemannian curvature tensor with respect to the Hilbert-Schmidt inner product
\begin{equation}
\label{eqndecomp}
\mbox{Riem}_g=\frac{R_g}{2n(n-1)}g\odot g+\frac{1}{n-2}\Big(Ric_g-\frac{R_g}{n}g\Big)\odot g+W_g,
\end{equation}
where  $\odot$ stands for the Kulkarni-Nomizu product of symmetric bilinear forms.

In 4 dimension, consider the following curvature $4$-form 
\begin{equation}\label{eqform}
\Omega_g=\left(Q_g+\frac{1}{4}|W_g|^2
\right)dvol_g.  
\end{equation}
Observe its close relation
to the Pfaffian, defined as $$\mbox{Pfaff}_g=\left(16 Q_g +4W^{abcd}W_ {abcd}-\frac{8}{3}\Delta_g R_g\right)dvol_g.$$  

 Using the  Gauss-Bonnet-Chern  formula and \eqref{eqndecomp} we have that
\begin{eqnarray*}
32\pi^{2}\chi(M)&=&\int_{M}\Big(|\mbox{Riem}_g|_g^{2}-4|Ric_g|_g^2+R_g^2\Big)dvol_g\\
&=&\int_{M}(4Q_g+|W_g|^2)dvol_g\\
&=&4\int_{M}\Omega_g.
\end{eqnarray*}

Next, we show that given a 4-form $\omega$ satisfying $$\int_M\omega=8\pi^2\chi(M),$$  we  find  a metric $\tilde g$ that satisfies $\omega=\Omega_{\tilde{g}}.$
This  new metric is obtained  pointwise conformal to $g.$

In order to proceed we need some preliminaries definitions. 
Recall that in four dimension, the {\it Paneitz-operator}  is a 4-th order differential operator defined by 
 $$P_gu=\Delta^2_gu+div_g \left(\frac{2}{3}R_gg -2 Ric_g\right)du,$$
where $d$ is the differential (acting on functions).
$P_g$ it is conformally invariant. Indeed, performing the conformal change of metric $\tilde{g}=e^{2\varphi}g$ we  get that 
$
P_{\tilde{g}}=e^{-4\varphi}P_g.
$
In this sense,  the transformation law  by conformal metric of the Paneitz operator represents an analogue of the Laplace-Beltrami operator.  
Moreover, it is well known that, as well as the $Q$-curvature,  $P_g$ is natural, that is, 
$\phi^*P_g = P_{\phi^*g}$ for all smooth diffeomorphism $\phi: M\to M,$ and self-adjoint with respect to the $L^2$-scalar product (see e.g. \cite{R},\cite{GZ},\cite{FG}). 

Now we are in position to prescribe curvature 4-forms in dimension four.

\begin{theorem}[Theoorem \ref{curvform}]
Let $(M,g)$ be a compact, connected, orientable Riemannian 4-manifold such that  $\ker P_g = \{const\}.$  Given any $4$-form $\omega$  
that satisfies $$
\int_{M}\omega=8\pi^2\chi(M),
$$
then there exist a metric pointwise conformal to $g$ such that $\omega$  is a curvature 4-form.
\end{theorem}

\begin{proof} 
The proof consists in seeking a metric $\tilde g=e^{2\varphi}g,$ or more precisely  a function $\varphi,$ in order to realize a given $\omega$ as $\Omega_{\tilde{g}}.$ First we recall that for   pointwise conformal metrics one has 
\begin{equation}\label{conf}
P_{g}\varphi + Q_{g} = Q_{\tilde{g}}e^{4\varphi},
\end{equation}
Thus
\begin{eqnarray}
\Omega_{\tilde{g}}&=&\left(Q_{\tilde{g}}+\frac{1}{4}|W_{\tilde{g}}|^2\right)dvol_{\tilde{g}}\\
&=&\Omega_g+P_g\varphi\; dvol_{g},\nonumber
\end{eqnarray}
where we have used that 
$$dvol_{\tilde{g}}=e^{-4\varphi}dvol_{g}\quad\mbox{and}\quad e^{4\varphi}\tilde{g}(W_{\tilde{g}},W_{\tilde{g}})=g(W_{g},W_{g}).$$

Thus, solve the linear equation 
\begin{equation}\label{solve}
\Omega_{\tilde{g}}-\Omega_g=P_g\varphi\; dvol_{g},
\end{equation}
 for some $\varphi,$ is equivalent to realize the 4-form $\omega$ as $\Omega_{\tilde{g}}.$  Moreover, observe that \eqref{solve} can be rewritten as
\begin{equation}\label{star}
P_g\varphi =*(\Omega_{\tilde{g}}-\Omega_g),
\end{equation}
  where $*$ stands for the Hodge star operation with respect to $g.$ 

Taking into account that $\int_M(\Omega_{\tilde{g}}-\Omega_g)=0$ and that $P_g$ is self-adjoint with $\ker P_g = \{const\}.$ It follows from elliptic theory that \eqref{star} has a unique solution up to an additive constant.
\end{proof}

\begin{remark}
 Although, $P_g$ and $\kappa_P$ are  conformal invariant objects on four manifolds, they provide some interesting  geometric information. Indeed, if a manifold of non-negative Yamabe invariant $Y(g)$ satisfies also $\kappa_P\geq0,$ then $\mbox{ker } P_g$ consists only of the constant functions and $P_g$ is a non-negative operator. Thus,  instead of assuming $P_g$ with trivial kernel,  one may suppose $ Y(g)\geq0$ and  $\kappa_P\geq0.$
\end{remark}

Graham-Jenne-Mason-Sparling  \cite{GJMS} have defined  a family of conformally invariant
 operators $P_{k,g}$ (in odd dimensions, $k$ is any positive integer, while in dimension $n$ even, $k$ is a positive integer no more than $\frac{n}{2}$), whose leading term is $\Delta^n_g,$ that are high-order analogues to the Laplace-Beltrami operator and to the Paneitz
operator for high dimensional compact manifolds. As the case treated here, these operators, the so-called of  GJMS operators,  have  associated curvature invariants $Q_{k,g}$. For more detail, see   \cite{GJMS}, \cite{CGJP},\cite{ND} and \cite{R}. Furthermore,  it was proved in \cite{BGP} that given a closed locally conformally flat manifold $(M, g)$ of even dimension $n$,  we have that
$$
C_n\int_MQ_{k,g}dV=\chi(M),
$$
where $C_n=\frac{1}{((n-2)!!)|\mathbb{S}^{n-1}|}$ ($|\mathbb{S}^{n-1}|$ denotes the volume of the standard (n-1)-sphere of radius $1$). Hence, the methods of Theorem \ref{curvform} apply equally well, with minor modifications, in order to prescribe curvature n-forms of locally conformally falt manifolds using $Q_{k,g}$and $P_{k,g}.$




\bibliographystyle{amsplain}

\begin{thebibliography}{10}

\bibitem{BFR}P. Baird, A. Fardoun, R. Regbaoui, \textit{Prescribed $Q$-curvature on manifolds of even dimension,} J. Geom. Phys. 59 (2) (2009) 221-233.

\bibitem{BE} M. Berger,  D. Ebin, \textit{Some decompositions of the space of symmetric tensors on a Riemannian
manifold,} J. Differential Geometry 3(1969), 379-392.

\bibitem{B}T.P. Branson, \textit{Differential operator scanonically associated to a conformal structure,} Math.Scand. 57 (1985), 293-345.

\bibitem{BGP}T. Branson, P. Gilkey and J. Pohjanpelto,  \textit{Invariants of locally conformally flat manifolds,} Trans. Amer. Math. Soc., 347 (1995), pp. 939-953.

\bibitem{Br}S. Brendle, \textit{Global existence and convergence for a higher order flow in conformal geometry,} Ann. Math. (2) 158 (2003), 323-343.

\bibitem{Br2} S. Brendle.  \textit{Convergence of the $Q$-curvature flow on $\mathbb S^4$}, Adv. Math. 205 (2006), 1-32.

\bibitem{CGJP} Y. Canzani, R. Gover, D. Jakobson, and R. Ponge. \textit{ Conformal invariants from nodal sets. I. Negative eigenvalues and curvature prescription.} Int. Math. Res. Not. IMRN, (9):2356-2400, 2014. With an appendix by Gover and Andrea Malchiodi.

\bibitem{CEOY} S.Y.A. Chang, M. Eastwood,  B. Ørsted, P. Yang,  \textit{ What is $Q$-curvature?.} Acta Appl. Math. 102(2-3) (2008), 119-125 

\bibitem{CY}S.Y.A. Chang, P.C. Yang, \textit{Extremal metrics of zeta function determinants on 4-manifolds,} Ann. of Math. 142 (1995) 171-212.



\bibitem{CGY} S.-Y.A.Chang, M.Gursky and P.Yang, \textit{Remarks on a fourth order invariant in conformal geometry,} Aspects of Mathematics, HKU. 353-372.

\bibitem{CR} H. Chtioui , A. Rigane, \textit{ On the prescribed $Q$-curvature problem on $\mathbb S^n$,} J. Funct. Anal. 261,  (2011) 2999-3043.

\bibitem{DM} Z. Djadli, A. Malchiodi, \textit{Existence of conformal metrics with constant $Q$-curvature,} Ann. of Math. 168 (2008), no.3, 813-858.

\bibitem{DR} P.Delano\"e, F. Robert.\textit{On the local Nirenberg problem for the $Q$-curvatures,}Pacific J. Math., 231 (2007), 293-304.


\bibitem{FG} C. Fefferman, C.R. Graham. \textit{$Q$-curvature and Poincar\'e metrics.} Math. Res. Lett. 9 (2002), 139-151.

\bibitem{FM2} A. Fischer, J. Marsden, \textit{Linearization stability of nonlinear partial differential equations,} Proc. Sympos. Pure Math., vol. 27, Amer. Math. Soc, Providence, R.I., 1975, pp. 219-263.

\bibitem{F-M2} A. Fisher, Mardsen, \textit{Linearization stability of nonlinear partial differential equations,}
Proc. Symp. Pure Math. 27, Part 2, 219-263 (1975).


\bibitem{GZ} C.R. Graham, M. Zworski. \textit{Scattering matrix in conformal geometry.} Invent. Math. 152 (2003), 89-118.

\bibitem{GJMS} C. R. Graham, R. Jenne, L. J. Mason and G. A. J. Sparling.  \textit{Conformally
invariant powers of the Laplacian. I. Existence.} J. London Math. Soc. (2)
46 no. 3 (1992), 557-565

\bibitem{G} M. Gursky, \textit{The principal eigenvalue of a conormally invariant differential operator, with an application to semilinear
elliptic PDE,} Comm. Math. Phys. 207 (1999) 131-147.

\bibitem{H}L. H\"ormander,  \textit{The Analysis of Linear Partial Differential Operators III: Pseudo-Differential Operators.} Springer-Verlag, 2007 [1985], ISBN 978-3-540-49937-4

\bibitem{KBOOK}J. Kazdan, \textit{ Prescribing the Curvature of a Riemannian Manifold,} Amer. Math. Soc., 1984 
(CBMS Regional Conference Series 57).

\bibitem{KW} J. Kazdan,  F. Warner, \textit{Curvature functions for compact 2-manifold}. Ann. of Math.,
99 (1974), 14-47.

\bibitem{KW4}J. Kazdan,  F. Warner, \textit{Curvature functions for open 2-manifold}. Ann. of Math.,
99 (1974),203-219.

\bibitem{KW1}  J. Kazdan,  F. Warner,\textit{ Existence and conformal deformation of metrics with prescribed Gaussian and scalar
curvature.} Ann. of Math. (2), 101 (1975), 317-331.

\bibitem{KW2} J. Kazdan,  F. Warner,\textit{Scalar Curvature and conformal deformation of
Riemannian structure.} J. Differ. Geom. 10 (1975) 113-134.

\bibitem{KW3} J. Kazdan,  F. Warner,\textit{ A direct approach to the determination of Gaussian
and scalar curvature functions.} Invent. Math. 28 (1975) 227-230.

\bibitem{LO} T. Levy, Y. Oz, \textit{Liouville Conformal Field Theories in Higher Dimensions,} (2018)[arXiv:1804.02283. [hep-th]].


\bibitem{LY}Y.-J. Lin, W. Yuan. \textit{Deformations of $Q$-curvature I}. Calc. Var. Partial Differential Equations, 55(4):Paper No. 101, 29, 2016.

\bibitem{MST}A. Malchiodi, M. Struwe, \textit{$Q$-curvature flow on $\mathbb S^4$,} J. Differential Geom. 73 (2006) 1-44

\bibitem{MS} V. G. Maz\'ya, T. O., Shaposhnikova,   \textit{Theory of multipliers in spaces of differentiable functions,} Monographs and Studies in Mathematics, 23, Pitman, Boston, MA, 1985.


\bibitem{N} Y. Nakayama, \textit{Canceling the Weyl anomaly from a position-dependent coupling,} Phys. Rev. D 97 (2018) no.4, 045008 doi:10.1103/PhysRevD.97.045008 [arXiv:1711.06413 [hep-th]].

\bibitem{ND}C.B. Ndiaye, \textit{Constant $Q$-curvature metrics in arbitrary dimension.} J. Funct. Anal. 251 (2007),
no. 1, 1-58.

\bibitem{R}  F. Robert. \textit{Admissible $Q$-curvatures under isometries for the conformal GJMS operators,} Nonlinear elliptic partial differential equations, Contemp. Math., vol. 540, Amer. Math. Soc.,Providence, RI, 2011, pp. 241-259.

\bibitem{WW} N. Wallach, F. Warner, \textit{Curvature forms for  2-manifolds}. Proc. Amer. of Math. Soc. 25, (1970), 712-713.

\bibitem{WX} J. Wei, X. Xu, \textit{On conformal deformations of metrics on $\mathbb S^n$}, J. Funct. Anal. 157 (1998) 292-325.


\end{thebibliography}

\end{document}